\documentclass{amsart}
\usepackage[utf8]{inputenc}
\usepackage{slashed}
\usepackage{textcomp} % provide lots of new symbols
\usepackage{amsmath,amstext,amsopn,amssymb, amsfonts, amsthm, mathtools}  % Better maths support & more symbols
\usepackage{mathrsfs}
\usepackage{graphicx}
\usepackage{xcolor}
\usepackage{geometry} % see geometry.pdf on how to lay out the page. There's lots.
\usepackage{array} % for better arrays (eg matrices) in maths
\usepackage{textcomp} % provide lots of new symbols
\usepackage[all,cmtip]{xy} % commutative diagrams
\usepackage{bbm}
\usepackage{varioref}
\usepackage{float}
\usepackage{pdfsync}

\usepackage[pdftex,bookmarks,colorlinks,breaklinks]{hyperref}  % PDF hyperlinks, with coloured links
\definecolor{dullmagenta}{rgb}{0.4,0,0.4}   % #660066
\definecolor{darkblue}{rgb}{0,0,0.4}
\definecolor{darkgreen}{rgb}{0,0.4,0}
\hypersetup{linkcolor=darkblue,citecolor=blue,filecolor=dullmagenta,urlcolor=darkblue} % coloured links
\newtheorem{theorem}{Theorem}[section]
\newtheorem*{theorem*}{Theorem}
\newtheorem{lemma}[theorem]{Lemma}
\newtheorem*{lemma*}{Lemma}
\newtheorem{proposition}[theorem]{Proposition}

\theoremstyle{definition}

\theoremstyle{remark}

\newtheorem{question*}[theorem]{Question}

\numberwithin{equation}{section}

\def\XXint#1#2#3{{\setbox0=\hbox{$#1{#2#3}{\int}$}
     \vcenter{\hbox{$#2#3$}}\kern-.5\wd0}}

\makeindex

\begin{document}

\newcommand{\todo}[1]{{\bf \color{red} [TODO: #1]}}

\title{On the embedding of $A_1$ into $A_\infty$}

\author{Guillermo Rey}
\address{Department of Mathematics, Michigan State University, East Lansing MI 48824-1027}
\email{reyguill@math.msu.edu}
\thanks{}

\begin{abstract}
    We give a quantitative embedding of the Muckenhoupt class $A_1$ into $A_\infty$. In particular, we show how
    $\epsilon$ depends on $[w]_{A_1}$ in the inequality which characterizes $A_\infty$ weights:
    \[
      \frac{w(E)}{w(Q)} \leq \biggl( \frac{|E|}{|Q|} \biggr)^\epsilon,
    \]
    where $Q$ is any dyadic cube and $E$ is any subset of $Q$. This embedding yields a sharp reverse-H\"older inequality
    as an easy corollary.
\end{abstract}

\maketitle

\section{Introduction}

The purpose of this article is to give a quantitative version of the classical embedding between Muckenhoupt classes
\begin{equation} \label{Embedding}
    A_1 \hookrightarrow A_{\infty}.
\end{equation}

The class $A_1$ is defined to be all weights $w \geq 0$ for which $Mw \leq Cw$ for some $C$, where
\[
    Mf(x) = \sup_{P \ni x} \frac{1}{|P|}\int_P |f(y)| \, dy
\]
is the uncentered Hardy-Littlewood maximal operator (here the supremum is taken over cubes with sides parallel to the coordinate axes).

The class $A_\infty$ is defined to be all weights $w \geq 0$ for which there exists a constant $C$ and an exponent $\epsilon > 0$ such that
\[
    \frac{w(E)}{w(P)} \leq C \biggl( \frac{|E|}{|P|} \biggr)^{\epsilon}
\]
for all cubes $P$ and all subsets $E \subseteq P$. Another common way to define this class is to introduce the so-called Fujii-Wilson
$A_\infty$ characteristic:
\[
  [w]_{A_\infty} := \sup_Q \frac{1}{w(Q)} \int_Q M(w\mathbbm{1}_Q) \, dx,
\]
where the supremum ranges over cubes with sides parallel to the axes. The class of $A_\infty$ weights is the collection of all weights for which
$[w]_{A_\infty}$ is finite.

These two definitions can be shown to be equivalent. In particular, one can give a quantitative version of the first:
\[
  [w]_{A_\infty'} := \inf\Biggl\{ a>0:\,  
    \frac{w(E)}{w(P)} \leq C \biggl( \frac{|E|}{|P|} \biggr)^{1/a}
    \quad \text{ for all cubes $P$ and all measurable subsets $E \subseteq P$}
\Biggr\}.
\]

With these definitions one can show that $[w]_{A_\infty} \sim_d [w]_{A_\infty'}$. The easy direction is $[w]_{A_\infty} \lesssim_d [w]_{A_\infty'}$,
to prove the reverse inequality one can use the sharp reverse-H\"older estimate found in \cite{Hytonen2012}.

We are interested in this form of the $A_\infty$ characteristic because it is the one which is used some recent proofs
of the weighted weak-type inequality for Calder\'on-Zygmund operators (see \cite{DomingoSalazar2015}), so it may yield some insights
into the sharpness of such estimate. In fact, in the Bellman-function approach to the sharpness of this weak-type estimate, one has a very similar
problem but with one extra difficulty; the problem treated in this article is the one which appears if this extra difficulty
is removed, see \cite{Nazarov2015} for more details.

% To see that these definitions are equivalent, let $Q$ be a cube and consider the sets
% \[
%   E_\lambda = \{x \in Q:\, M(w\mathbbm{1}_Q)(x) > \lambda\}
% \]
%
% One has
% \[
%   |E_\lambda| \leq \frac{C_d}{\lambda}w(E_{\lambda}),
% \]
% as in the usual proof of the Hardy-Littlewood maximal inequality using the Vitali covering lemma.
%
% Upon dividing by $w(Q)$ and applying the definition of $[w]_{A_\infty'}$ we obtain
% \[
%   \frac{|E_\lambda|}{w(Q)} \leq \frac{C_d}{\lambda} \frac{w(E_\lambda)}{w(Q)} \leq \frac{C_d}{\lambda} \biggl( \frac{|E_\lambda|}{|Q|} \biggr)^\frac{1}{a},
% \]
% for all $a > [w]_{A_\infty'}$. Hence
% \[
%   \biggl( \frac{|E_\lambda|}{|Q|} \biggr)^{1 - \frac{1}{a} } \leq \frac{C_d}{\lambda} \frac{w(Q)}{|Q|}.
% \]
%
% If we now use the trivial estimate $|E_\lambda| \leq |Q|$ together with the above inequality we obtain:
% \begin{align*}
%   \frac{1}{w(Q)} \int_Q M(w \mathbbm{1}_Q) \, dx &= \frac{1}{w(Q)}\int_0^\infty |E_\lambda| \, d\lambda \\
%   &= \frac{1}{w(Q)}\biggl( \lambda_0|Q| + \widetilde{C}_d \biggl(\frac{w(Q)}{|Q|}\biggr)^{a'-1}\int_{\lambda_0}^\infty \lambda^{-a'} \, d\lambda\biggr) \\
%   &\lesssim_d a
% \end{align*}
% if one chooses $\lambda_0 = \frac{w(Q)}{|Q|}$.
% To see the other direction, i.e.: $[w]_{A_\infty'} \lesssim_d [w]_{A_\infty}$, one can use the sharp reverse H\"older inequality in \cite{Hytonen2012}, we
% omit the standard computations.

It is a well-known fact that every weight in $A_1$ is also in $A_\infty$; here we give a quantitative version of this embedding.

We will actually work with a wider class of weights, the dyadic $A_p$ weights.
To state the result, let us fix a way to quantify exactly how a weight lies in dyadic $A_1$. Let $P$ be a cube in $\mathbb{R}^d$, we define the $A_1^d(P)$ characteristic of a weight $w \geq 0$ to be
\[
    [w]_{A_1^d(P)} := \operatorname{ess\,sup}_{x \in P} \frac{M^{\text{dyadic}}_P w(x)}{w(x)},
\]
where $M^{\text{dyadic}}_P$ is the dyadic maximal operator localized to $P$:
\[
    M^d_P f(x) = \sup_{R \in \mathcal{D}(P)} \langle |f| \rangle_R \mathbbm{1}_R(x).
\]
Here we are denoting by $\mathcal{D}(P)$ the collection of all dyadic subcubes of $P$, and the average of a function $f$ over a set $E$ by
\[
    \langle f \rangle_E := \frac{1}{|E|} \int_E f(x) \, dx.
\]
Also, we denote the characteristic function of a set $E$ by $\mathbbm{1}_E$.

We define the (non-dyadic) $A_1$ characteristic similarly:
\[
    [w]_{A_1(P)} := \operatorname{ess\,sup}_{x \in P} \frac{M_P w(x)}{w(x)},
\]
where $M_P$ is the uncentered Hardy-Littlewood maximal operator where the cubes are constrained to lie inside $P$.

The classical way to prove \eqref{Embedding} proceeds by using the \emph{reverse H\"older inequality} of Coifman-Fefferman \cite{Coifman1974} (see \cite{Hytonen2012}
for a recent sharp reverse H\"older inequality valid in a very general context): for any weight $w \in A_p$ we have
\[
    \langle w^q \rangle_P \leq C \langle w \rangle_P^q,
\]
for some exponent $q > 1$ depending on $w$. Indeed, let $C_{\text{RH}}$ be the best constant in the above inequality (which will depend on $q$ and on how $w$ lies in $A_p$), then:
\begin{align*}
    w(E) &= \int_P w \mathbbm{1}_E \\
         &\leq \Bigl( \int_P w^q \Bigr)^{1/q} |E|^{1/q'} \\
         &\leq C_{\text{RH}}^{1/q} w(P) \Bigl( \frac{|E|}{|P|} \Bigr)^{1/q'}.
\end{align*}
For (non-dyadic) $A_1$ weights the most quantitative version of the reverse H\"older inequality was given by \cite{Vasyunin2003} in dimension one. Using the results of \cite{Vasyunin2003} one obtains
\[
  \frac{w(E)}{w(P)} \leq \frac{a}{a-1} \Bigl( \frac{|E|}{|P|} \Bigr)^{\frac{1}{a[w]_{A_1(P)}}}
\]
for all $a > 1$, so one can get arbitrarily close to the exponent $\frac{1}{[w]_{A_1}}$ at the cost of a multiplicative constant. The results in \cite{Vasyunin2003} are, however, valid only for non-dyadic $A_p$ weights,
which behave much better in terms of sharp constants; also \cite{Vasyunin2003} is valid only in dimension $1$.

In \cite{Melas2005a} A. Melas showed that, for dyadic $A_1$ weights, one has
\[
    \Bigl\langle (M^{\text{dyadic}}w)^p \Bigr\rangle_P \leq C(p,[w]_{A_1^d}) \langle w \rangle_P^p,
\]
for all $p$ such that
\[
    1 \leq p < \frac{\log(2^d)}{\log\Bigl( 2^d - \frac{2^d-1}{[w]_{A_1^d}} \Bigr)},
\]
and where $C(p,[w]_{A_1^d})$ is a constant which blows-up as $p$ tends to the endpoint above.

Following the same steps as before, this implies an inequality of the form
\[
    \frac{w(E)}{w(P)} \leq C_{\epsilon} \Bigl( \frac{|E|}{|P|} \Bigr)^{\epsilon}
\]
for all $\epsilon$ such that
\[
    0 \leq \epsilon < -\frac{\log\Bigl( 1 - \frac{2^d-1}{2^d [w]_{A_1^d}} \Bigr)}{d\log 2} := \epsilon([w]_{A_1^d},d),
\]
and where $C_\epsilon$ is a constant which blows-up as $\epsilon$ tends to the endpoint $\epsilon([w]_{A_1^d},d)$.

It was of interest whether one could achieve an estimate with the endpoint $\epsilon([w]_{A_1^d},d)$,
and this was answered positively by A. Os\c{e}kowski in \cite{Oscekowski2013}, where he proved the following weak-type estimate:
\begin{equation}\label{Osc}
    \frac{1}{|P|}\Bigl| \Bigl\{ x \in P: \, M^{\text{dyadic}} w(x) > 1 \Bigr\} \Bigr| \leq \langle w \rangle_P^p
\end{equation}
for all $p$ such that
\[
    1 \leq p \leq \frac{\log(2^d)}{\log\Bigl( 2^d - \frac{2^d-1}{[w]_{A_1^d}} \Bigr)}.
\]
This estimate, coupled with H\"older's inequality for Lorentz spaces yields
\[
    \frac{w(E)}{w(P)} \leq C_{\epsilon(Q,d)}\Bigl( \frac{|E|}{|P|} \Bigr)^{\epsilon(Q,d)}
\]
for all weights $w$ with $[w]_{A_1^d} \leq Q$,
thus settling the endpoint question of whether a decay rate of $(|E|/|P|)^{\epsilon(Q,d)}$ could be achieved. However, note that H\"older's inequality for Lorentz spaces (when used in this way) has a constant which explodes
when $p \to 1$ which in this case implies that the constant $C_{\epsilon(Q,d)}$ will blow-up as $Q \to \infty$.

In this article we improve this conclusion by directly computing the function
\[
    \mathbb{B}(x,y,m) = \sup \frac{w(E)}{|P|},
\]
where the supremum is taken over all sets $E \subseteq P$ with $|E|/|P| = x$,
and all dyadic $A_1$ weights $w$ with $[w]_{A_1^d(P)} \leq Q$, $\langle w \rangle_P = y$
and $\operatorname{ess\,inf}_{z \in P}w(z) = m$.

This is what is commonly called the \textit{Bellman function} associated with the problem. It is an extremal object which
controls the way in which the parameters evolve when ``concatenating'' several weights and sets together.

We can already give an upper bound for $\mathbb{B}(\cdot, Q,1)$:
\begin{equation}\label{SmoothEstimate}
    \mathbb{B}(x,Q,1) \leq \widetilde{f}(x) := Qx^{\epsilon(Q,d)}.
\end{equation}
This shows that the decay rate deduced from Os\c{e}kowski's estimate can be achieved with a uniform constant as $Q \to \infty$ (note that the constant $Q$ cancels when estimating $\frac{w(E)}{w(P)}$).
Observe also that this recovers the result of Os\c{e}kowski when one takes $w$ instead of its maximal function in \eqref{Osc}, which can be interpreted as a
weak-type reverse H\"older inequality. Indeed, assume without loss of generality that $|P| = \operatorname{ess\,inf} w = 1$ and
let $E_\lambda = \{x \in P: w(x) > \lambda\}$, then our estimate will show (see \eqref{DefOfM}) that
\[
    w(E_\lambda) \leq Q \Bigl( \frac{w(P)-1}{Q-1} \Bigr) \Bigl( |E_\lambda| \frac{Q-1}{w(P)-1} \Bigr)^{\epsilon(Q,d)}.
\]
So integrating $w$ over this set yields
\[
    \lambda |E_\lambda|^{1-\epsilon(Q,d)} \leq (\langle w \rangle_P -1)^{1-\epsilon(Q,d)} \Bigl( \frac{Q}{(Q-1)^{\epsilon(Q,d)}} \Bigr) \leq \langle w \rangle_P.
\]
Or, in other words,
\[
    \|w\|_{L^{p,\infty}} \leq \int_P w(x) \, dx
\]
for the same $p$'s as in \eqref{Osc}.

However, the function $\mathbb{B}(\cdot, Q,1)$ is, surprisingly, slightly better. Indeed if we define
$f(x) = \mathbb{B}(x,Q,1)$,
then our main result shows that $f$ is the piecewise-linear interpolation
of the function $\widetilde{f}$ evaluated at the points $2^{-dk}$ for $k \in \mathbb{N}$.
\begin{figure}[H]
    \caption{Plots of $f$ and $\widetilde{f}$}
    \label{Figure:Difference}
    \includegraphics[scale=0.8]{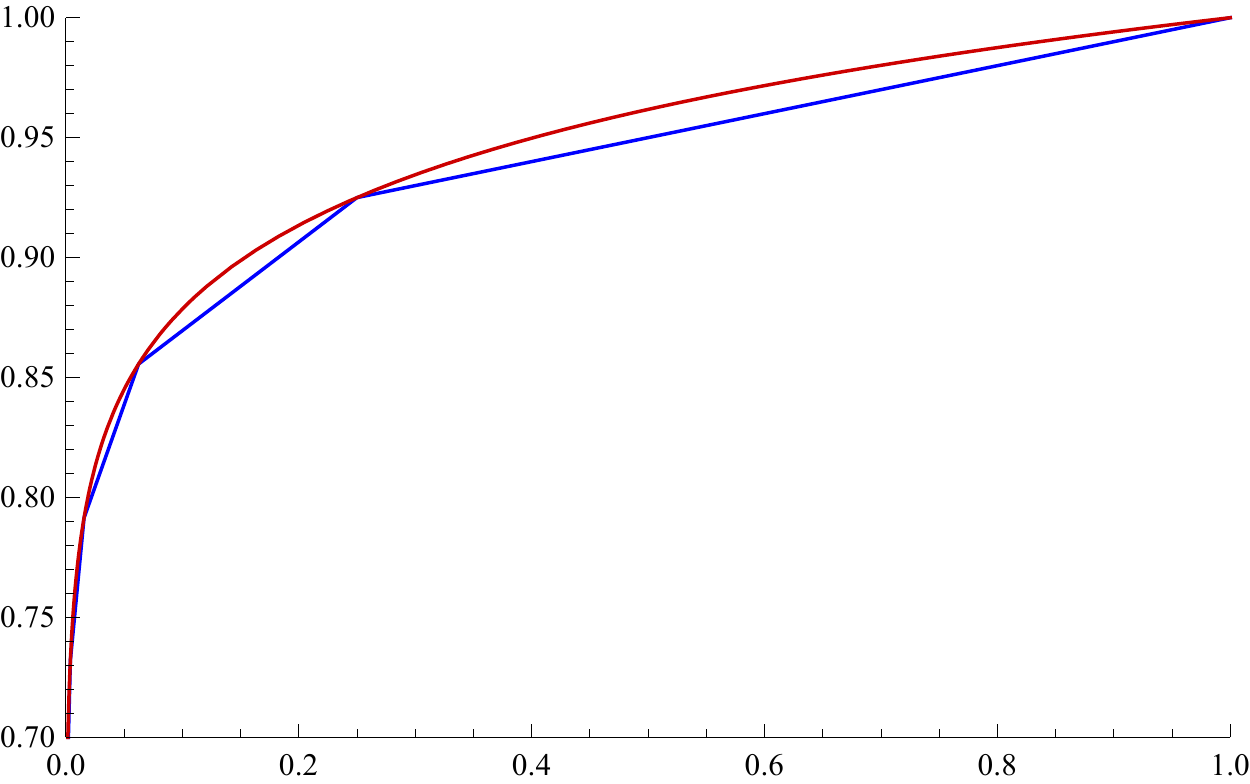}
\end{figure}

In \autoref{Figure:Difference} we show a normalized section of the plot (the values are divided by $Q$) of the functions $f$ and $\widetilde{f}$ with $Q = 10$ and in dimension two.

The precise form of $\mathbb{B}$ is given in the following theorem, which is the main result of this article.
\begin{theorem}
  The function $\mathbb{B}$ defined above is has the form
  \[
    \mathbb{B}(x,y,m) = m\cdot
\begin{cases}
        x+y/m-1 &\text{if }y/m \leq 1+(Q-1)x \\
        \frac{y/m-1}{Q-1}f\Bigl( x \frac{Q-1}{y/m-1} \Bigr) &\text{if } y/m \geq 1 + (Q-1)x
    \end{cases} \Biggr\}.
  \]
\end{theorem}

\subsection{Organization}
The article is organized as follows: in section \ref{BellmanFunctionSetting} we cast the problem as one of finding a certain Bellman function, then in section \ref{BellmanFunctionLowerBound} we give
a lower bound for the Bellman function; we also describe the structure of the maximizers. In section \ref{CandidateBellmanFunctionSatisfiesMI} we show that the lower bound found in the previous section
is also an upper bound, hence showing that the function found is the actual Bellman function.

\section{Acknowledgements}
I wish to thank Professor Alexander Volberg for many invaluable discussions regarding
Bellman functions. I also am greatly indebted to Professor Ignacio Uriarte-Tuero
without whom this project could not have happened. I also benefited very much from discussions
about this result with Professors David Cruz-Uribe, Leonid Slavin, and Vasily Vasyunin.
Finally, I would like to thank Professor Cristina Pereyra for organizing the New Mexico
Analysis Seminar, which provided the perfect environment for many stimulating
discussions related to this work.

\section{The Bellman function approach} \label{BellmanFunctionSetting}
Define, as in the introduction, the function
\[
    \mathbb{B}_P(x,y,m) = \sup\Bigl\{ \frac{w(E)}{|P|}: E \subseteq P,\, [w]_{A^d_1(P)} \leq Q \text{ such that } |E| = x|P|,\, \langle w \rangle_P = y,\, m = \operatorname{ess\,inf} w \Bigr\}.
\]
By translation and dilation invariance, the function $\mathbb{B}_P$ is independent of $P$, so we suppress the index $P$ from $\mathbb{B}$ from now on.

The domain, which will be denoted by $\Omega_{\mathbb{B}}$ is:
\begin{align*}
    0 &\leq x \leq 1 \\
    0 < m &\leq y \leq Qm.
\end{align*}

In this section we cast finding $\mathbb{B}$ as a minimization problem. We will follow the Bellman function method, see for example \cite{Nazarov1999},
\cite{Vasyunin2003} or \cite{Slavin2008}, and \cite{Oscekowski2013} for an approach closer to ours.

The function $\mathbb{B}$ satisfies the following \emph{Main Inequality}
\begin{equation} \label{MainInequalityForB}
    \mathbb{B}(x,y,m) \geq \Bigl\langle \mathbb{B}(x_i,y_i,m_i) \Bigr\rangle,
\end{equation}
where $\langle x_i \rangle = x$, $\langle y_i \rangle = y$, $\min m_i = m$, $(x_i, y_i, m_i) \in \Omega$, and $(x,y,m) \in \Omega$.
In inequality \eqref{MainInequalityForB}, and for the rest of the article, we use the notation
\[
    \langle \xi_i \rangle := \frac{1}{n} \sum_{i=1}^n \xi_i,
\]
whenever $\{\xi_i\}$ is a discrete sequence of $n$ numbers; usually $n$ will be obvious from the context so we will omit its dependence.

We can see \eqref{MainInequalityForB} by combining almost-extremizers defined on the first-generation dyadic subcubes of $P$ into one on the whole cube $P$.

We also have the \emph{obstacle condition}
\[
    \mathbb{B}(1,y,y) = y,
\]
which is just the observation that if $E = P$ almost everywhere, then $\langle \mathbbm{1}_E w \rangle_P = \langle w \rangle_P$.

From the definition of $\mathbb{B}$ we have the homogeneity property
\begin{equation} \label{HomogeneityForB}
    \mathbb{B}(x,\lambda y, \lambda m) = \lambda \mathbb{B}(x,y,m).
\end{equation}

If we find a nonnegative function $B$ defined in $\Omega_\mathbb{B}$ and which satisfies the main inequality and the obstacle condition above,
then $\mathbb{B} \leq B$. This is a typical fact whose proof we omit, but the reader can consult \cite{Oscekowski2013} for a proof in a similar case.

The homogeneity condition will let us assume that $m = 1$ in \eqref{MainInequalityForB}:
\begin{proposition}
    If a function $B$ defined on $\Omega_{\mathbb{B}}$ satisfies the main inequality \eqref{MainInequalityForB} with $m = 1$ and the homogeneity property \eqref{HomogeneityForB},
    then it must also satisfy the main inequality for all $m > 0$.
\end{proposition}
\begin{proof}
    This is just the observation that the domain of $B$ is invariant under simultaneous dilations of the variables $y$ and $m$.
\end{proof}

We want to find a set of necessary and sufficient conditions for $B$ to satisfy the main inequality, but which are simpler to verify. To this end, let us first prove \emph{necessary} conditions
that any such $B$ must satisfy.

The following Lemma is simple but important in what follows. It tells us that, in order to exploit \eqref{MainInequalityForB}, we should strive to minimize the variables $m_i$ as much as possible.
We will let $N := 2^d$ for the rest of the article.
\begin{lemma}
    Any function $B$ satisfying \eqref{MainInequalityForB} is decreasing in $m$. More precisely: assume $(x,y,m_1)$ and $(x,y,m_2)$ are two points in $\Omega_{\mathbb{B}}$ with $m_1 \leq m_2$, then
    \begin{equation} \label{BIsDecreasingInM}
        B(x,y,m_1) \geq B(x,y,m_2).
    \end{equation}
\end{lemma}
\begin{proof}
    Let $x_i = x$ and $y_i =y$ for all $1 \leq i \leq 2^d := N$. Also, let
    \[
        \widetilde{m}_i =
            \begin{cases}
                m_1 &\text{if } i = 1 \\
                m_2 &\text{if } i > 1.
            \end{cases}
    \]
    Then the points $(x_i,y_i,\widetilde{m}_i)$ are all in $\Omega$. Also, $\langle x_i \rangle = x$ and $\langle y_i \rangle = y$. Since $m_1 \leq m_2$ we also have that $\min( \widetilde{m}_i ) = m_1$, so
    using \eqref{MainInequalityForB}:
    \begin{align*}
        B(x,y,m_1) \geq \frac{1}{N}B(x,y,m_1) + \frac{N-1}{N} B(x,y,m_2),
    \end{align*}
    which after rearranging yields \eqref{BIsDecreasingInM}.
\end{proof}

The following Lemma follows directly from the main inequality \eqref{MainInequalityForB}.
\begin{lemma}
    For any fixed $m > 0$, the function $(x,y) \mapsto B(x,y,m)$ is concave.
\end{lemma}
\begin{proof}
    This is just the observation that the domain $\Omega$ is convex, together with \eqref{MainInequalityForB} with $m_i = m$.
\end{proof}

Now we are able to make the first reduction in \eqref{MainInequalityForB} (after the trivial one of setting $m=1$):
\begin{proposition} \label{MainInequalityUnderHomogeneity}
    Suppose $B$ is a nonnegative function defined in $\Omega_{\mathbb{B}}$ and which satisfies the obstacle condition, \eqref{HomogeneityForB}, and \eqref{BIsDecreasingInM}. If $B$ satisfies
    \begin{equation} \label{MainInequalityForBv2}
        B(x,y,1) \geq \Bigl\langle B\Bigl(x_i, y_i, \max\Bigl(1, \frac{y_i}{Q} \Bigr) \Bigr) \Bigr\rangle
    \end{equation}
    for all $N$-tuples of points $(x_i, y_i)$ satisfying
    \begin{align}
        0 &\leq x_i \leq 1, \quad \text{and} \quad \langle x_i \rangle = x, \\
        1 &\leq y_i, \quad \min(y_i) \leq Q, \quad \text{and} \quad \langle y_i \rangle = y,
    \end{align}
    then we must have that $B = \mathbb{B}$.
\end{proposition}
\begin{proof}
    The above conditions make \eqref{MainInequalityForBv2} certainly necessary. To see that it is sufficient, take any $N$-tuple $(x_i,y_i,m_i)$ of points in $\Omega_{\mathbb{B}}$
    satisfying
    \[
        \langle x_i \rangle = x, \quad \langle y_i \rangle = y \quad \text{and} \quad \min( m_i ) = 1.
    \]

    Consider now the alternative $N$-tuple formed by $(x_i,y_i,\widetilde{m}_i)$, where
    \begin{align*}
        \widetilde{m}_i &=
            \begin{cases}
                \frac{y_i}{Q} &\text{if } y_i \geq Q \\
                1 &\text{otherwise}.
            \end{cases} \\
            &= \max\Bigl( 1, \frac{y_i}{Q} \Bigr).
    \end{align*}

    These points all lie in $\Omega_{\mathbb{B}}$ and moreover they still satisfy the condition
    \[
        \min( \widetilde{m}_i ) = 1.
    \]

    However, by inequality \eqref{BIsDecreasingInM} we have
    \[
        B\Bigl( x_i, y_i, \max\Bigl( 1, \frac{y_i}{Q} \Bigr) \Bigr) \geq B(x_i, y_i, m_i ).
    \]
\end{proof}

This proposition is useful because it allows us to ``almost'' eliminate the third variable from our analysis. The reason that we used the word ``almost'' comes from the fact that we still have the extraneous condition that
$\min(y_i) \leq Q$, which is an effect of having $\min(m_i) = 1$. We now proceed to eliminate this condition too.

Suppose that of the $N$ points $(x_i, y_i)$, there are exactly $N-k$ of them for which $y_i \geq Q$. Then, after possibly reordering the inequality (which we can do without loss of generality), the right hand side of \eqref{MainInequalityForBv2}
becomes
\[
    \frac{1}{N}\Bigl( \sum_{i=1}^k B(x_i,y_i,1) + \sum_{i=k+1}^N B\Bigl(x_i,y_i, \max\Bigl( \frac{y_i}{Q} \Bigr) \Bigr) \Bigr)
\]
which can be written, after applying the homogeneity property \eqref{HomogeneityForB}, as
\begin{equation*}
    \frac{1}{N}\Bigl( \sum_{i=1}^k B(x_i,y_i,1) + \sum_{i=k+1}^N \frac{y_i}{Q}B(x_i,Q, 1)\Bigr).
\end{equation*}

So, verifying \eqref{MainInequalityForBv2} reduces to just showing that $B$ is concave in $(x,y)$, decreasing in $m$, and that for each $1 \leq k \leq N-1$
\begin{equation} \label{MainInequalitiesForBv3}
    B(x,y,1) \geq \frac{1}{N}\Bigl( \sum_{i=1}^k B(x_i,y_i,1) + \sum_{i=k+1}^N \frac{y_i}{Q}B(x_i,Q, 1)\Bigr)
\end{equation}
for all $(x,y)$ and all $(x_i,y_i)$ as in Proposition \eqref{MainInequalityUnderHomogeneity}, with the additional assumption that $y_i \geq Q$ for $k \geq k+1$.

The next proposition allows us to just consider the case where $k=N-1$ in the above inequality.
\begin{proposition}
    Let $M$ be a nonnegative function defined on $\Omega$ and which satisfies that
    \begin{enumerate}
        \item $M$ is concave.
        \item The function $t \mapsto tM(x,y/t)$ is decreasing.
        \item For all $(x,y)$ and all $(\widetilde{x},\widetilde{y})$ in $\Omega$ we have
            \begin{equation} \label{MainInequalityforM}
                M(x,y) \geq \frac{N-1}{N}M(\widetilde{x},\widetilde{y}) + \frac{Ny - (N-1)\widetilde{y}}{QN}M(Nx-(N-1)\widetilde{x},Q)
            \end{equation}
            whenever $Nx-(N-1)\widetilde{x} \geq 0$ and $Ny-(N-1)\widetilde{y} \geq Q$.
    \end{enumerate}

    Then, defining $B$ by homogeneity as in \eqref{HomogeneityForB}:
    \[
        B(x,y,m) = m M(x,y/m),
    \]
    yields a function which satisfies the conditions of Proposition \ref{MainInequalityUnderHomogeneity}
\end{proposition}
\begin{proof}
    First of all note that, by the above discussion, we just need to find $M$ satisfying the conditions (1), (2) and
    \[
        M(x,y) \geq \frac{1}{N}\Bigl( \sum_{i=1}^k M(x_i,y_i,1) + \sum_{i=k+1}^N \frac{y_i}{Q}M(x_i,Q, 1)\Bigr),
    \]
    where the average of $x_i$ is $x$, the average of $y_i$ is $y$ and all $y_i \geq Q$ for $i \geq k+1$.

    Also, note that \eqref{MainInequalityforM} is just the case of \eqref{MainInequalityForBv2} with $k=N-1$. So, in what follows we assume $k < N-1$.

    Fix all points $(x_i, y_i)$ for $i \leq k$ and consider the collection $\mathcal{V}$ of all vectors $\vec{y} = (y_{k+1}, \dots, y_N)$ with $y_i \geq Q$ for $k \geq k+1$ and satisfying.
    \[
        \frac{1}{N}\sum_{i=k+1}^N y_i + \frac{1}{N}\sum_{i=1}^k y_i = y.
    \]

    We can write this condition as
    \[
        \widehat{y} := \frac{1}{N-K}\sum_{i=k+1}^N y_i = \frac{Ny-\sum_{i=1}^k y_i}{N-k} = \frac{Ny-k\widetilde{y}}{N-k},
    \]
    where we have defined $\widetilde{y} = \frac{1}{k} \sum_{i=1}^k y_i$.

    It is an easy exercise to verify that
    \[
        \frac{1}{N} \sum_{i=k+1}^N \frac{y_i}{Q}M(x_i,Q) \leq \frac{1}{QN}\sum_{i=k+1}^N b_i M(x_i,Q),
    \]
    where $b_i$ are defined by
    \[
        b_i =
        \begin{cases}
            Q &\text{if } i \neq i_{\text{max}} \\
            (N-k)\widehat{y}-Q(N-k-1) &\text{if } i=i_{\text{max}},
        \end{cases}
    \]
    and where $i_{\text{max}}$ is defined to be the index which maximizes $M(x_i,Q)$ for $i \geq k+1$.

    Observe that the vector $(b_{k+1}, \dots, b_N) \in \mathcal{V}$, so we can assume that $y_i = b_i$ for $i \geq k+1$. But then, we can reorganize the
    inequality to put all of the terms except one (the one with $i_{\text{max}}$) on the first summation. Writing it this way makes it evident that it really
    was a particular example of the inequality with $k = N-1$.
\end{proof}

\section{Finding the Bellman function} \label{BellmanFunctionLowerBound}

In this section we give a lower bound $M$ for $\mathbb{M}$, and in the next section we will show that this lower bound is also an upper bound and hence that $M = \mathbb{M}$.

First recall that
\[
    t \mapsto \mathbb{M}(x,y/t)
\]
is non-increasing and therefore that $\mathbb{M}(1,y) \geq y$ (here we are using the obstacle $\mathbb{M}(1,1) = 1$. Since $\mathbb{M}(0,1) \geq 0$, we now can extend this bound to the subdomain $y \leq 1 + (Q-1)x$ to get:
\[
    \mathbb{M}(x,y) \geq x+y-1 \quad \forall (x,y) \in \Omega:\, y \leq 1 + (Q-1)x.
\]

We will now give a lower bound for $\mathbb{M}$ in the rest of the domain. The idea is to use inequality \eqref{MainInequalityforM} setting the number $Nx - (N-1)\widetilde{x}$ to be as large as possible,
within the domain that we know, and then iterate.

More precisely let $x_0 = 1$, observe that if $Nx - (N-1)\widetilde{x} = x_0$, then
\[
    \widetilde{x} = \frac{Nx - x_0}{N-1}.
\]
Clearly we need $x \geq 1/N$ for $\widetilde{x}$ to be in the domain, so we set $x = \frac{1}{N}$.
We will also make $\widetilde{y}$ as small as possible, which means $\widetilde{y} = 1$.

Putting it all together we obtain, using \eqref{MainInequalityforM} with $x= \frac{1}{N}$ and $y = Q$:
\[
    \mathbb{M}\Bigl(\frac{1}{N},Q\Bigr) \geq \frac{NQ-(N-1)}{NQ}\mathbb{M}(x_0,Q) = Q \Bigl( 1 - \frac{N-1}{NQ} \Bigr).
\]

Now we iterate this procedure. Set $x = x_{k+1} = \frac{x_k}{N}$, $y = Q$, $\widetilde{y} = 1$ and $\widetilde{x} = 0$, then \eqref{MainInequalityforM} gives
\[
    \mathbb{M}(x_{k+1},Q) \geq ( 1 - \frac{N-1}{NQ} \Bigr) \mathbb{M}(x_k,Q),
\]
so
\[
    \mathbb{M}(N^{-k},Q) \geq Q \Bigl( 1 - \frac{N-1}{NQ} \Bigr)^k.
\]

Between $x_{k+1}$ and $x_k$ we know that $M(\cdot, Q)$ is concave, so $\mathbb{M}$ must certainly be at least linear in these intervals. Now, since $\mathbb{M}(0,1) \geq 0$, we can
also extend this bound by homogeneity and get the upper bound
\begin{align*}
    \mathbb{M}(x,y) \geq \frac{y-1}{Q-1}\mathbb{M}\Bigl( x\frac{Q-1}{y-1}, Q \Bigr) \geq \frac{y-1}{Q-1}f\Bigl( x \frac{Q-1}{y-1} \Bigr)
\end{align*}
for $y-1 \geq x$. Here, $f$ is the piecewise linear function defined on $[0,1]$ by linearly interpolating the points
\[
    f(x_k) = Q\Bigl( 1 - \frac{N-1}{NQ} \Bigr)^k
\]
between $x_{k+1}$ and $x_k$, \autoref{Figure:Difference} shows what $f$ typically looks like.

Putting it all together, we get
\begin{equation} \label{DefOfM}
    \mathbb{M}(x,y) \geq
    \begin{cases}
        x+y-1 &\text{if }y \leq 1+(Q-1)x \\
        \frac{y-1}{Q-1}f\Bigl( x \frac{Q-1}{y-1} \Bigr) &\text{if } y \geq 1 + (Q-1)x.
    \end{cases} \Biggr\} =: M(x,y).
\end{equation}

The way we proved these bounds also shows how one would construct pairs of weights $w$ and sets $E$ showing that $\mathbb{M}$ is at least the promised lower bound. We now give a detailed description
of these examples.

\subsection{Explicit extremizers}
Let's start with examples corresponding to the line $(1,y)$ with $y \in [1,Q]$. To get the bound $\mathbb{M}(1,y) \geq y$ we used the main inequality keeping all the parameters fixed except one of the $m_i$'s.
So let us repeat the proof, but now with actual weights. Fix a cube $P$ and let $P_1, \dots P_N$ be its dyadic children. Define $w_i(x) = 1$ for all $i$ and all $x \in P_i$ except for $i = N$, for which we define
$w_i(x) = 1 + N(y-1)$ for all $x \in P_N$.
Now define $w(x) = w_i(x)$ for all $x \in P_i$; clearly $\operatorname{ess\,inf}_{x \in P}w(x) = 1$ and $\langle w \rangle_P = y$. Now, since $x = 1$, we should set $E = P$.
The pair $(w,E)$ is clearly contained in the supremum in the definition of $\mathbb{B}(1,y,1) = \mathbb{M}(1,y)$ and so
\begin{equation} \label{ExampleForEasyBoundary}
    \mathbb{M}(1,y) \geq \frac{w(P)}{|P|} = y
\end{equation}
for this particular choice of $w$. Of course, any weight with $\langle w \rangle_P = y$ would also have been sufficient since $x = 1$.

Examples for weights and sets corresponding to points $(x,y)$ on the rest of the domain are more complicated. We will start by constructing examples along the line $y = Q$.

The way we proved that $\mathbb{M}(\frac{1}{N},Q) \geq Q(1-\frac{N-1}{NQ})$ was by using \eqref{MainInequalityforM} with $\widetilde{x} = 0$, $\widetilde{y} = 1$, $x = \frac{1}{N}$ and $y = Q$. Similarly, we
got the bound $\mathbb{M}(x_{k+1},Q) \geq (1-\frac{N-1}{NQ}) \mathbb{M}(x_k,Q)$ by using \eqref{MainInequalityforM} with $\widetilde{x} = 0$, $\widetilde{y} = 1$, $x = \frac{1}{N^{k+1}}$ and $y = Q$. Looking back
at how we got \eqref{MainInequalityforM}, we see that we combined $N-1$ trivial weight-set pairs (the pairs $(w \equiv 1, E = \emptyset)$) with an example coming from
\[
    \mathbb{B}\Bigl(\frac{1}{N^k},N(Q-1)+1,N - \frac{N-1}{Q}\Bigr).
\]
We then used homogeneity to translate this to an example which would extremize
\[
    \mathbb{M}\Bigl( \frac{1}{N^k}, Q \Bigr),
\]
but having lost a factor slightly larger than one.

We can trace back these steps with the following lemma:
\begin{lemma} \label{BoundaryConstructionLemma}
    Let $P$ be a cube in $\mathbb{R}^d$.
    Given a pair $(w,E)$ where $w$ is a dyadic $A_1$ weight with $[w]_{A_1} \leq Q$ and with $\langle w \rangle_P = Q$, $\operatorname{ess\,inf}_{z \in P}w(z) = 1$, and $\langle \mathbbm{1}_E \rangle_P = x$, there exists a pair
    $(\widetilde{w}, \widetilde{E})$ where $\widetilde{w}$ is another dyadic $A_1$ weight with $[w]_{A_1^d} \leq Q$ and
    with $\langle \widetilde{w} \rangle_P$, $\operatorname{ess\,inf}_{z \in P}\widetilde{w}(z) = 1$, and $\langle \mathbbm{1}_{\widetilde{E}} \rangle_P = x/N$ for which
    \[
        \frac{\widetilde{w}(\widetilde{E})}{|P|} \geq \Bigl( 1 - \frac{N-1}{NQ} \Bigr) \frac{w(E)}{|P|}.
    \]

    Moreover, the set $\widetilde{E}$ is entirely contained in one of the dyadic subcubes of $P$ and $\widetilde{w}$ is identically $1$ on the complement of $\widetilde{E}$.
\end{lemma}
\begin{proof}
    As before, enumerate the children of $P$ by $P_1, \dots, P_N$. We start by translating and dilating $(w,E)$ to the subcube $P_1$, we do this with the obvious linear change of variables. We then multiply
    the weight we just constructed by the constant $\frac{NQ-(N-1)}{Q}$. Let us call this new weight $w_1$.
    Clearly $\operatorname{ess\,inf}_{z \in P_1}w_1(z) = \frac{NQ-(N-1)}{Q} \geq 1$ and $\langle w_1 \rangle_{P_1} = NQ-(N-1)$. Now define $w_i(z) = 1$ for all $z \in P_i$ and each $i \geq 2$ and combine all of these weights into one:
    $\widetilde{w}(z) = w_i(z)$, for all $z \in P_i$. This new weight is a dyadic $A_1$ weight with $[\widetilde{w}]_{A_1^d} \leq Q$.

    The set $E$ is just translated and dilated to $P_1$ using the same change of variables used to define $w_1$. Now $\widetilde{E}$ is just a scaled copy of $E$
    living in $P_1$, so we of course have $\langle \mathbbm{1}_{\widetilde{E}} \rangle = x/N$.
    % With $E$ we do almost the same: we translate and dilate it to $P_1$; let us call this new set $E_1$. We have $\langle \mathbbm{1}_{E_1} \rangle = x$. Define $\widetilde{E}$ to be just $E_1$
    % (so $\mathbbm{1}_{\widetilde{E}}(z) = \mathbbm{1}_{E_1}(z)$).

    We assert that this new pair $(\widetilde{w},\widetilde{E})$ satisfies the promised estimate. Indeed (assuming without loss of generality that $|P| = 1$):
    \begin{align*}
        \widetilde{w}(\widetilde{E}) &= \frac{1}{N}\Bigl( (N-1)w_2(\widetilde{E}) + w_1(\widetilde{E})  \Bigr) \\
                                     &= \frac{1}{N}w_1(\widetilde{E}) \\
                                     &= \Bigl( 1 - \frac{N-1}{NQ} \Bigr) w(E),
    \end{align*}
    which is what we wanted.
\end{proof}

Given a cube $P$ and a pair $(w,E)$ as in Lemma \ref{BoundaryConstructionLemma}, we define
\[
    T(w) = \widetilde{w},
\]
where $\widetilde{w}$ is the weight constructed in the proof of Lemma \ref{BoundaryConstructionLemma}. Similarly, we define $S(E) = \widetilde{E}$.

With this lemma at hand we can now describe the structure of the examples which show that $\mathbb{M}(N^{-k},Q) \geq Q(1-\frac{N-1}{NQ})^k$.
\begin{lemma} \label{BoundaryPointsDiscrete}
    Let $P$ be any cube and let $w_0$ be the weight constructed when proving \eqref{ExampleForEasyBoundary} (but any weight with $\langle w_0 \rangle_P = Q$, $\operatorname{ess\,inf}_{z \in P}w_0(z) = 1$, and with $[w]_{A_1^d} = Q$
    will work as well).

    Define the weights $w_k$ and the sets $E_k$ inductively by
    \[
        w_{k+1} = Tw_k \quad \text{and} \quad E_{k+1} = SE_k,
    \]
    where $E_0 = P$.

    Then $w_k$ is an $A_1^d$ weight with $[w]_{A_1^d} = Q$, $\langle w_k \rangle_P = Q$, $\operatorname{ess\,inf}_{z \in P}w_k(z) = 1$, $\langle \mathbbm{1}_{E_k} \rangle_P = N^{-k}$ and
    \[
        \frac{w_k(E_k)}{|P|} = Q\Bigl( 1 - \frac{N-1}{NQ} \Bigr)^k.
    \]
\end{lemma}
\begin{proof}
    The proof is just to iteratively apply Lemma \ref{BoundaryConstructionLemma}.
\end{proof}

It remains to extend the examples to the rest of the domain. But recall that the bound we gave for $\mathbb{M}$ on the rest of the domain was obtained by linear interpolation, so we just need to combine examples that
have already been constructed.

The following lemma shows how to combine two pairs $(w_0,E_0)$ and $(w_1,E_1)$ into one:
\begin{lemma}
    Let $P$ be a cube and let $(w_0,E_0)$ and $(w_1,E_1)$ be two pairs. Assume $w_0$ and $w_1$ are both dyadic $A_1$ weights with $[w_i]_{A_1^d} \leq Q$, and also:
    \[
        \langle \mathbbm{1}_{E_i} \rangle_P = x_i, \quad \langle w_i \rangle_P = y_i, \quad \operatorname{ess\,inf}_{z \in P} w_i(z) = 1.
    \]

    Then, for any $\lambda \in [0,1]$ we can construct a pair $\mathcal{C}_\lambda((w_0,E_1), (w_1,E_1)) = (w,E)$, where $w$ is a dyadic $A_1$ weight with $[w]_{A_1^d} \leq Q$,
    \[
        \langle \mathbbm{1}_E \rangle_P = x, \quad \langle w \rangle_P = y, \quad \operatorname{ess\,inf}_{z \in P}w(z) = 1,
    \]
    and
    \[
        \frac{w(E)}{|P|} = (1-\lambda)\frac{w_0(E_0)}{|P|} + \lambda \frac{w_1(E_1)}{|P|},
    \]
    where
    \[
        x = (1-\lambda)x_0 + \lambda x_1 \quad \text{and} \quad y = (1-\lambda)y_0 + \lambda y_1.
    \]
\end{lemma}
\begin{proof}
    Note that, at least when $\lambda$ is a dyadic rational, repeated applications of the Main Inequality give exactly these dynamics. So we should follow the proof of the Main Inequality, whose
    meaning is to show what happens when one combines pairs $(w_i,E_i)$ defined on the dyadic children of a cube into one pair $(w,E)$ on the whole cube.

    There is a slight technicality: if one applies this combination procedure a finite number of times, one can only prove this lemma in the case where $\lambda$ is a dyadic rational, but
    we can still prove this lemma with a limiting argument.

    Let $b_i$ be the digits of $\lambda$ when written in binary:
    \[
        \lambda = \sum_{i=1}^\infty b_i 2^{-i}
    \]
    (it does not matter which of the possible binary representations one uses).

    Fix the cube $P$ and let $R$ be any of its dyadic subcubes. Define $S_{P \to R}$ to be the linear change of variables which maps $P$ to $R$.

    Given a cube $P$ let $P_1, \dots P_N$ be a fixed enumeration of its first-generation children, this ordering will be fixed throughout the proof (in the sense that we will use the same ordering on
    every other cube, which we obtain by translating and dilating the original ordering).

    The idea is to split the subcubes of $P$ and on half of them put a translated and dilated copy of either $(w_0,E_0)$ or $(w_1,E_1)$, depending on the binary digit of the current step. We apply the same procedure
    on each of the remaining cubes (but now with the next digit).

    More precisely, let $\operatorname{ch}(P)$ be the first-generation dyadic subcubes of $P$ and define $\mathcal{H}_{\pm}^1(P)$ to be the subset of $\operatorname{ch}(P)$ consisting of the first or second half the dyadic children, i.e.:
    \[
        \mathcal{H}_{-}^1(P) = \{ P_1, \dots, P_{2^{d-1}} \} \quad \text{and} \quad \mathcal{H}_{+}^1(P) = \{P_{2^{d-1}+1}, \dots, P_{2^d} \}.
    \]

    We inductively define $\mathcal{H}_{\pm}^{j+1}(P)$ as follows:
    \[
        \mathcal{H}_{\pm}^{j+1}(P) = \bigcup_{R \in \mathcal{H}_{+}^{j}(P)} \mathcal{H}_{\pm}(R).
    \]

    We define the weight $w$ by
    \[
        w(x) = \sum_{j=1}^\infty \sum_{R \in \mathcal{H}^j_{-}(P)}\Bigl( (1-b_j)S_{P \to R}w_0(x) + b_j S_{P \to R}w_1(x) \Bigr).
    \]
    This definition can be pictured as follows: we put a certain weight (either $w_0$ or $w_1$ depending on $b_j$) in half of the dyadic children of $P$,
    then we again place a copy of $w_0$ or $w_1$ on one half of the first generation children of each of the remaining cubes from the previous step. This process is
    repeated inifinitely many times (thus exhausting the full cube $P$), and with either $w_0$ or $w_1$ in each step depending on the binary digit expansion of
    the number $\lambda$.

    Similarly, we define the set $E$ by
    \[
        \mathbbm{1}_E(x) = \sum_{j=1}^\infty \sum_{R \in \mathcal{H}^j_{-}(P)}\Bigl( (1-b_j)S_{P \to R}\mathbbm{1}_{E_0}(x) + b_j S_{P \to R}\mathbbm{1}_{E_1}(x) \Bigr).
    \]

    One can now check that this pair satisfies the required properties; see \cite{Rey2014} for a very similar construction.
\end{proof}

With this Lemma, we can now express the structure of the examples on the line $y = Q$ of $\Omega$ which lie between the points with coordinates $x=N^{-k}$. Indeed, let $(w_k,E_k)$ be the weight-set pair
constructed by Lemma \ref{BoundaryPointsDiscrete}. Then for any $x \in (N^{-k-1},N^{-k})$ we have
\[
    (w_x,E_x) := \mathcal{C}_{\lambda}((w_{k+1},E_{k+1}),(w_{k},E_k)),
\]
where
\[
    x = (1-\lambda)N^{-k-1} + \lambda N^{-k}.
\]

To extend to the rest of $\Omega$, let $(x,y) \in \Omega$ with $y < Q$. First assume that $y \leq 1 + (Q-1)x$; then we should use the previous Lemma with boundary on $x = 1$. Indeed
let
\[
    (w,E) = \mathcal{C}_{\lambda}((\mathbbm{1},\emptyset),(w^y,P)),
\]
where $\lambda = 1 + \frac{y-1}{x}$ and where $w^y$ is any dyadic $A_1$ weight with $[w]_{A_1^d} \leq Q$, $\langle w^y \rangle_P = y$ and $\operatorname{ess\,inf}_{z \in P}w(z) = 1$. This pair
clearly satisfies all the required estimates.

Now suppose that $y \geq 1 + (Q-1)x$ and let $(w_\ast,E_\ast)$ be the pair we just constructed on the line $y = Q$ with $x$-coordinate $x \frac{Q-1}{y-1}$. Then
\[
    (w,E) = \mathcal{C}_\lambda \bigl( (\mathbbm{1}, \emptyset), (w_\ast,E_\ast) \bigr),
\]
with $\lambda = x \frac{Q-1}{y-1}$ also satisfies all the required estimates.

\section{Verifying the Main Inequality} \label{CandidateBellmanFunctionSatisfiesMI}

We now have to show that the function $M$ that we found in the previous section satisfies all the required conditions which, we recall, are:
\begin{enumerate}
    \item $M$ is concave.
    \item The function $t \mapsto t M(x,y/t)$ is nonincreasing.
    \item For all $(x,y) \in \Omega$ and all $(\widetilde{x},\widetilde{y})$ in $\Omega$ with $\widetilde{x} \leq x$ and $Ny - (N-1)\widetilde{y} \geq Q$, we have
    \begin{equation}\label{IdontKnowYet}
        M(x,y) \geq \frac{N-1}{N}M(\widetilde{x},\widetilde{y}) + \frac{Ny - (N-1)\widetilde{y}}{NQ}M(Nx-(N-1)\widetilde{x},Q).
    \end{equation}
\end{enumerate}

It will be convenient to examine the function $f$, in particular observe that
\[
    f'(x) = (N\eta)^k,
\]
where $\eta = 1 - \frac{N-1}{NQ}$, whenever $x \in (N^{-k-1}, N^{-k})$.

The ratio $\eta N > 1$ whenever $Q > (N-1)/N$, which is always the case since $Q \geq 1$, hence $f$ is concave. Since $f$ is concave, it follows that $M$ must also be concave, since $M$ is
just the extension of $f$ by homogeneity to the subdomain of $\Omega$ which lies above the diagonal $y = 1+(Q-1)x$, and below this line the function is just the plane $z = x+y-1$.
A brief check now shows that $M$ is
indeed concave in $\Omega$. This proves (1).

Now we will show that the function
\[
    t \mapsto tM(x,y/t)
\]
is decreasing, thus proving (2).

To show this, note that we just need to prove $yM_y \geq M$ wherever $M$ is differentiable. This obviously holds for $y < 1+ (Q-1)x$, so it suffices to assume $y > 1+(Q-1)x$.
By homogeneity, we can translate this condition to one for $f$:
\[
    \frac{y}{Q-1}f\Bigl( x \frac{Q-1}{y-1} \Bigr) - \frac{xy}{y-1}f'\Bigl( x \frac{Q-1}{y-1} \Bigr) \geq \frac{y-1}{Q-1}f\Bigl( x \frac{Q-1}{y-1}\Bigr).
\]

Let $u = x \frac{Q-1}{y-1}$, then this inequality becomes
\[
    \frac{1}{u}f(u) - y f'(u) \geq 0
\]
for all $u \in [0,1]$ and all $y \in [1,Q]$. Since $f$ is increasing, this inequality is strongest when $y = Q$, so it suffices to show
\[
    f(u) \geq Qu f'(u).
\]

Recall that $f$ is piecewise linear, so let $u_0 = N^{-k-1}$ and $u_1 = N^{-k}$ and assume $u \in (u_0,u_1)$. The above inequality now becomes
\[
    f(u_0) + (u-u_0)f'(u_0+) \geq Qu f'(u_0+).
\]
Thus, we can reduce to showing
\[
    \frac{f(u_0)}{f'(u_0+)} \geq u_0 + (Q-1)u_1.
\]
But an easy computation, using the value of $f'$ computed before, yields that this inequality is equivalent to
\[
    \eta \geq 1 - \frac{N-1}{NQ},
\]
which is precisely the value of $\eta$ so we are done. This shows (2).

Finally, we are left with verifying (3). To do this we will construct a sequence of functions $M_k$ defined on $\Omega$, all of which satisfy (3) on a specific subset of $\Omega$. Define
\[
    \Omega_k = \{(x,y) \in \Omega:\, y \leq 1 + (Q-1)N^k x \}.
\]

\begin{figure}[!ht]
    \caption{Domains $\Omega_k$}
    \label{Figure:Domain2}
    \includegraphics[scale=0.6]{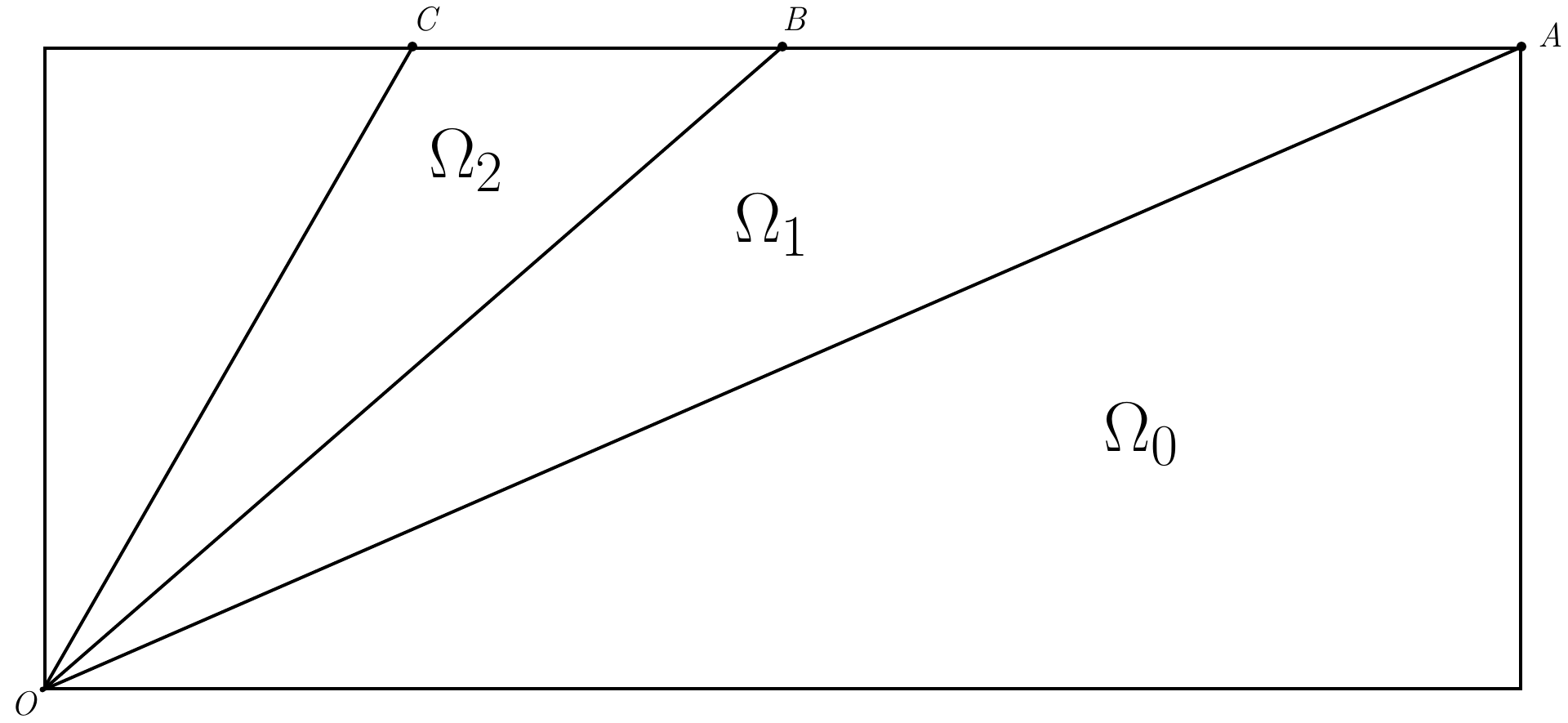}
\end{figure}
\autoref{Figure:Domain2} represents the first three of these domains (again, the diagram is not to scale). For example $\Omega_2$ is the subdomain of $\Omega$ which lies to the right of the line
joining $O$ and $C$.

We define $M_k$ to be the wedge formed by the $k$-th plane of $M$ on $\Omega\setminus\Omega_{k-1}$ and the $(k-1)$-th plane of $M$ on $\Omega_{k-1}$, that is:
\[ M_k(x,y) =
    \begin{cases}
        a_k x + b_k (y-1) & \text{if } (x,y) \in \Omega \setminus \Omega_{k-1} \\
        a_{k-1}x + b_{k-1} (y-1) & \text{if } (x,y) \in \Omega_{k-1}.
    \end{cases}
\]
where $M(x,y) = a_k x + b_k (y-1)$ on $\Omega_k \setminus \Omega_{k-1}$. One can give the explicit formulas for $a_k$ and $b_k$:
\[
    a_k = (N \eta)^k, \quad b_k = \eta^k.
\]

Obviously $M_0$ satisfies (3).

Fix any $(x,y) \in \Omega$, we can assume without loss of generality that $(x,y) \in \Omega_k$ for some $k$. Introduce the notation
\[
    x = \frac{N-1}{N} \widetilde{x} + \frac{1}{N}\widehat{x} \quad \text{and} \quad y = \frac{N-1}{N} \widetilde{y} + \frac{1}{N}\widehat{y}.
\]

Since $M$ is concave, we have that $M_k \geq M$ on $\Omega$ ($M_k$ is a ``supporting wedge'' of the graph of $M$).
Instead of (3) we will prove (under the same hypotheses)
\begin{equation} \label{StrongerMI}
    M_k(x,y) \geq \frac{N-1}{N}M_k(\widetilde{x},\widetilde{y}) + \frac{1}{N}\frac{\widehat{y}}{Q} M_k( \widehat{x}, Q),
\end{equation}
which, by the above remark, is a stronger statement.

We will first show that we can assume the point $(\widehat{x}, Q)$ to be in $\Omega_k$.
Indeed, suppose that $\widetilde{x}$ is so small that
$(\widehat{x},Q) \notin \Omega_k$, then
\begin{align*}
    \frac{\partial}{\partial \widetilde{x}}\Bigl( \text{Right hand side of }\eqref{StrongerMI} \Bigr) &= \frac{N-1}{N} a_k - \frac{N-1}{N} \frac{\widehat{y}}{Q}a_{k-1} \\
                                                                                                     &= \Bigl(\frac{N-1}{N}\Bigr) \Bigl( a_k - \frac{\widehat{y}}{Q}a_{k-1} \Bigr) \\
                                                                                                     &\geq \Bigl(\frac{N-1}{N}\Bigr) \Bigl( a_k - \frac{Ny-(N-1)}{Q}a_{k-1} \Bigr) \\
                                                                                                     &\geq \Bigl(\frac{N-1}{N}\Bigr) \Bigl( a_k - \frac{NQ-(N-1)}{Q}a_{k-1} \Bigr).
\end{align*}

Now recall that $a_k = (N \eta)^k$, so the partial derivative of the right hand side of equation \eqref{StrongerMI} is at least
\[
    \frac{N-1}{N}(N\eta)^{k-1}\Bigl( N\eta - \frac{NQ-(N-1)}{Q}\Bigr) = 0,
\]
so the right hand side is increasing, at least as long as $(\widehat{x},Q) \in \Omega_{k-1}$.

This allows us to assume that $\widetilde{x}$ is large enough to make $(\widehat{x},Q) \in \Omega_k$ (by continuity). Under this assumption the inequality becomes much easier since $M_k$ is
now being evaluated always on $\Omega_k$, and hence we can assume that $M_k$ itself is a plane. Now it is easy to check that the inequality is indeed true under these conditions.

To see this, observe that inequality \eqref{StrongerMI} can be written as:
\[
    ax + b(y-1) \geq \frac{N-1}{N}\bigl( a\widetilde{x} + b(\widetilde{y}-1) \bigr) + \frac{1}{N}\frac{\widehat{y}}{Q}\bigl( a\widehat{x} + b(Q-1) \bigr).
\]

We can reorganize this as:
\begin{align*}
    a\Bigl( x - \frac{N-1}{N} \widetilde{x} - \frac{1}{N}\frac{\widehat{y}}{Q} \widehat{x} \Bigr) + b\Bigl( y-1 - \frac{N-1}{N} \widetilde{y} + \frac{N-1}{N} - \frac{1}{N}\frac{\widehat{y}}{Q}(Q-1)  \Bigr) &\geq 0.
\end{align*}

This simplifies to showing
\[
    a\Bigl( \frac{\widehat{x}}{N} - \frac{\widehat{x}}{N}\frac{\widehat{y}}{Q} \Bigr) + b\Bigl( \frac{\widehat{y}}{NQ} - \frac{1}{N} \Bigr) \geq 0,
\]
which is equivalent to
\[
    \Bigl( \frac{\widehat{y}}{Q} -1 \Bigr)\bigl( b - a\widehat{x} \bigr) \geq 0.
\]

Since the assumptions force $\widehat{y}$ to be at least $Q$, we just need to check that $\widehat{x} \leq \frac{b}{a}$. But this is exactly the bound that is guaranteed from the considerations above
since $\frac{b}{a} = N^{-k}$.

\nocite{Vasyunin2003}
\bibliography{bibliography}{}

\begin{thebibliography}{10}

\bibitem{Coifman1974}
R.~R. Coifman and C.~Fefferman.
\newblock Weighted norm inequalities for maximal functions and singular
  integrals.
\newblock {\em Studia Math.}, 51:241--250, 1974.

\bibitem{DomingoSalazar2015}
C.~{Domingo-Salazar}, M.~T. {Lacey}, and G.~{Rey}.
\newblock {Borderline weak-type estimates for singular integrals and square
  functions}.
\newblock {\em Bulletin of the London Mathematical Society}, Dec. 2015.

\bibitem{Hytonen2012}
T.~Hyt{\"o}nen, C.~P{\'e}rez, and E.~Rela.
\newblock Sharp reverse {H}\"older property for {$A_\infty$} weights on spaces
  of homogeneous type.
\newblock {\em J. Funct. Anal.}, 263(12):3883--3899, 2012.

\bibitem{Melas2005a}
A.~D. Melas.
\newblock A sharp {$L^p$} inequality for dyadic {$A_1$} weights in
  {$\mathbb{R}^n$}.
\newblock {\em Bull. London Math. Soc.}, 37(6):919--926, 2005.

\bibitem{Nazarov2015}
F.~{Nazarov}, A.~{Reznikov}, V.~{Vasyunin}, and A.~{Volberg}.
\newblock {A Bellman function counterexample to the $A_1$ conjecture: the
  blow-up of the weak norm estimates of weighted singular operators}.
\newblock {\em ArXiv e-prints}, June 2015.

\bibitem{Nazarov1999}
F.~Nazarov, S.~Treil, and A.~Volberg.
\newblock The {B}ellman functions and two-weight inequalities for {H}aar
  multipliers.
\newblock {\em J. Amer. Math. Soc.}, 12(4):909--928, 1999.

\bibitem{Oscekowski2013}
A.~Os{\c{e}}kowski.
\newblock Sharp inequalities for dyadic {$A_1$} weights.
\newblock {\em Arch. Math. (Basel)}, 101(2):181--190, 2013.

\bibitem{Rey2014}
G.~Rey and A.~Reznikov.
\newblock Extremizers and sharp weak-type estimates for positive dyadic shifts.
\newblock {\em Adv. Math.}, 254:664--681, 2014.

\bibitem{Slavin2008}
L.~Slavin, A.~Stokolos, and V.~Vasyunin.
\newblock Monge-{A}mp\`ere equations and {B}ellman functions: the dyadic
  maximal operator.
\newblock {\em C. R. Math. Acad. Sci. Paris}, 346(9-10):585--588, 2008.

\bibitem{Vasyunin2003}
V.~I. Vasyunin.
\newblock The exact constant in the inverse {H}\"older inequality for
  {M}uckenhoupt weights.
\newblock {\em Algebra i Analiz}, 15(1):73--117, 2003.

\end{thebibliography}
\bibliographystyle{abbrv}

\end{document}